\documentclass[11pt,a4paper]{article}

\usepackage{amsmath}
\usepackage{amsfonts}
\usepackage{amssymb}
\usepackage{centernot}
\usepackage[top=3cm, bottom=3cm, left=2.8cm, right=2.8cm]{geometry}
\usepackage{enumitem}
\usepackage{hhline}
\usepackage{graphicx,fancybox}
\usepackage{tikz}
\usepackage{subcaption}
\usepackage{amsthm}
\usepackage{pgf}
\usepackage[font={small}]{caption}
\usepackage{etoolbox}
\usepackage{array, makecell, cellspace}
\usepackage{soul}

\DeclareMathAlphabet\mbc{OMS}{cmsy}{b}{n}

\setlist[enumerate]{label={(\roman*)},topsep=3pt}

\newcommand{\Id}{\operatorname{Id}}

\newcommand{\bs}{\boldsymbol}
\newcommand{\Hi}{\mathcal{X}}
\newcommand{\Rr}{\mathbb{R}}
\newcommand{\N}{\mathbb{N}}

\newcommand{\Ball}{\mathbb{B}}

\def\tto{\rightrightarrows}

\usepackage{titlesec}
\usepackage{titling}
\usepackage{anyfontsize}
\pretitle{\vspace{-1cm}\begin{center}\fontsize{16}{20}\bfseries\sffamily}
\pretocmd{\theequation}{\small}{}{}

\makeatletter
\def\th@plain{%
	\thm@notefont{}
	\itshape 
}
\def\th@definition{%
	\thm@notefont{}
	\normalfont 
}

\usepackage[colorlinks=true,linkcolor=blue,citecolor=blue,urlcolor=magenta]{hyperref}
\usepackage[capitalise]{cleveref}

\newtheorem{theorem}{Theorem}[section]

\newtheorem{corollary}[theorem]{Corollary}
\newtheorem{lemma}[theorem]{Lemma}
\newtheorem{fact}[theorem]{Fact}

\theoremstyle{definition}
\newtheorem{definition}[theorem]{Definition}

\newtheorem{remark}[theorem]{Remark}

\Crefname{fact}{Fact}{Facts}
\Crefname{assumption}{Assumption}{Assumptions}
\Crefname{enumi}{}{}

\numberwithin{equation}{section}

\title{Regularity of sets under a reformulation in a product space of reduced dimension}

\author{Rub\'en Campoy \thanks{Department of Mathematics, University of Alicante, Alicante, \textsc{Spain}. 
{E-mail:}~\href{mailto:ruben.campoy@ua.es}{ruben.campoy@ua.es}}}

\begin{document}
\graphicspath{{Python/}}
\maketitle

\begin{abstract}
Different notions on regularity of sets and of collection of sets play an important role in the analysis of the convergence of projection algorithms in nonconvex scenarios. While some projection algorithms can be applied to feasibility problems defined by finitely many sets, some other require the use of a product space reformulation to construct equivalent problems with two sets. In this work we analyze how some regularity properties are preserved under a reformulation in a product space of reduced dimension. This allows us to establish local linear convergence of parallel projection methods which are constructed through this reformulation.

\paragraph{\small MSC\,2020:} 47H05 $\cdot$ 47J25 $\cdot$ 49M27 $\cdot$ 65K10 $\cdot$ 90C30
\end{abstract}

\section{Introduction}

The so-called \emph{feasibility problem} asks for a point in the intersection of a family of sets $C_1,\ldots,C_r$ in an Euclidean space; that is,
\begin{equation}\label{eq:FP}
\text{Find} \quad x^*\in C_1\cap C_2\cap \cdots\cap C_r.
\end{equation}
Projection algorithms are widely employed methods for solving \eqref{eq:FP} whenever the individual projectors onto the sets can be easily computed. The \emph{method of alternating projections (MAP)} \cite{APvN} and the \emph{Douglas--Rachford algorithm (DR)} \cite{DR56,LM79} are well known projection algorithms originally devised for solving feasibility problems with two sets. While the former can be naturally extended for an arbitrary number of sets \cite{B65}, it is not so obvious for the case of DR (see, e.g., \cite[Section~3.3]{aragon2020ORclassroom}). Although there exist some cyclic versions of DR for finitely many sets \cite{BNP15,MT14}, these are not \emph{frugal} in the sense that some of the projectors are employed more than once at each iteration. In contrast, one can always apply \emph{Pierra's product space reformulation}~\cite{Pierra} to derive a frugal parallel DR-type projection algorithm embedded in the product Euclidean space~$\Hi^r$. This enlargement of the dimension of the ambient space has been called as $r$-\emph{fold lifting}. In general, reduced lifting is preferred as this leads to computational memory savings.

In the more general context of monotone inclusions, which include feasibility problems as particular cases, the impossibility of a frugal three-operator DR algorithm without lifting was proved in \cite{ryu}. In addition, the author showed that the minimal lifting for three operators is $2$-fold. That result has been recently generalized in \cite{MT21} for an arbitrary number of operators. Further, frugal splitting algorithms with minimal $(r-1)$-lifting have been independently proposed in \cite{campoy21,MT21}. By now, the analysis of splitting algorithms with reduced lifting has become a very active research topic; see, e.g.~\cite{aragonbot21,aragonmalitsky21,BSW22,bricenoarias,DDHT21,Tam22}.

While the convergence of projection algorithms is well understood when the sets are convex, they are also popular in nonconvex settings. In this framework, local linear convergence of the schemes is usually analyzed by assuming some regularity properties of the individual sets and of their intersection; see, e.g.,
\cite{BNP15,BPW14,DDHT21,DP18,DP19,HL13,K06,KLT18,LLM09,Phan16}. In this work we analyze how some of these properties are preserved through the product space reformulation with reduced dimension studied by the author in \cite{campoy21}. This trick reformulates problem \eqref{eq:FP} as an equivalent feasibility problem defined by two sets in the product space $\Hi^{r-1}$ while keeping the computability of the projectors. Thus, it allows for devising new projection algorithms with $(r-1)$-lifting from already existing two-sets methods. Although the reformulation was shown to be valid for not necessarily convex sets but rather proximinal, there was a lack of theoretical results from the perspective of the local convergence of projection algorithms. The aim of this work is to extend the analysis of the reformulation by establishing that super-regularity of the sets (see~\Cref{def:superreg_sets}), as well as linear regularity and strong regularity of their intersection (see~\Cref{def:reg_collect_sets}), are inherit by the new product sets in the reformulated problem.  Hence, the local linear convergence of the derived algorithms can be deduced assuming that those conditions hold for the original problem.

The structure of this manuscript is as follows. We collect some preliminary notions and results in \Cref{sec:prelim}. In \Cref{sec:ProdSpace} we revisit the product space reformulation with reduced dimension and we include our main result regarding the preservation of regularity properties. We apply our results in \Cref{sec:applic} to derive the local linear convergence of a parallel projection algorithm with reduced lifting, where we also include a numerical experiment to illustrate the result. Finally, some conclusions are drawn in \Cref{sec:Conc}.

\section{Preliminaries}\label{sec:prelim}

Throughout this paper, $\Hi$ is a Euclidean space endowed with inner product $\langle \cdot,\cdot \rangle$ and induced norm $\|\cdot\|$. The set of nonnegative integers is denoted by $\N$ and $\Ball(x;\delta)$ stands for the closed ball centered at $x\in\Hi$ with radius $\delta\geq 0$. Given a linear subspace $M\subseteq\Hi$ we denote by $M^\perp$ to its orthogonal complement, i.e., $M^\perp=\{u\in\Hi : \langle x,u\rangle=0, \forall x\in M\}$ .

\subsection{Projection mapping}

Given a nonempty set $C\subseteq\Hi$, the \emph{distance function} to $C$,  $d_C:X\to[0,+\infty)$, is given at $x\in\Hi$ by
$$d_C(x):=\inf_{c\in C}\|c-x\|.$$

\begin{definition}
The \emph{projection mapping} (or \emph{projector}) onto $C$ is the possibly set-valued operator $P_C:\Hi\tto C$ defined at each $x\in\Hi$ by
\begin{equation*}
P_C(x):=\left\{p\in C : \|x-p\|=\inf_{c\in C}\|c-x\|\right\}.
\end{equation*}
Any point $p\in P_C(x)$ is said to be a \emph{best approximation} to $x$ from $C$ (or a \emph{projection} of $x$ onto $C$).
\end{definition}

If $C$ is assumed to be closed, then it is \emph{proximinal}; i.e., a projection onto $C$ exists for every point in the space (see, e.g., \cite[Corollary 3.15]{BC17}). When $C$ is in addition convex, then $C$ is \emph{Chebyshev}; i.e., the projector $P_C(x)$ is single-valued for all $x\in\Hi$ (see, e.g., \cite[Remark~3.17]{BC17}).

We recall next some properties of the projector.
\begin{fact}\label[fact]{fact:PC}
Let $C\subseteq\Hi$ be nonempty. The following hold.
\begin{enumerate}[label=(\roman*)]
\item If $C$ is closed and convex, then $P_C$ is continuous.\label{fact:PC_cont}
\item If $C$ is a linear subspace, then $P_C$ is a linear mapping.\label{fact:PC_linear}
\end{enumerate}
\begin{proof}
\ref{fact:PC_cont}: See, e.g., \cite[Proposition~3.12]{BC17}. \ref{fact:PC_linear}: See, e.g., \cite[Theorem~5.13]{D01}. 
\end{proof}
\end{fact}

In general, no closed expression exists for the projector onto the intersection of two sets, in terms of the individual projectors. However, if one of the involved sets is an affine subspace with some additional intersection structure with another closed set, we can establish the following relation on the projectors.
\begin{fact}\label[fact]{fact:Pinter}
Let $C\subseteq\Hi$ be nonempty and closed and let $D\subseteq\Hi$ be an affine subspace. If $P_C(d)\cap D\neq \emptyset$, for all $d\in D$, then 
$$P_{C\cap D}(x)=P_C(P_D(x))\cap D,\quad  \forall x\in \Hi.$$
\begin{proof}
See~\cite[Lemma~2.10]{campoy21}.
\end{proof}
\end{fact}

\subsection{Normal cone}

Regularity notions shall be defined in terms of the \emph{(limiting) normal cone} to the sets.%

\begin{definition}
The \emph{(limiting) normal cone} to $C\subseteq\Hi$ at a point $x\in C$ is given by
\begin{equation*}
N_C(x):=\left\{\lim_{k\to\infty}\tau_k(z_k-x_k) : \tau_k\geq 0, x_k\to x \text{ and } z_k\in P_C(x_k), \text{for all } k\in\N \right\},
\end{equation*}
whereas we just set $N_C(x)=\emptyset$ for $x\not\in C$.
\end{definition}

For a closed and convex set $C\subseteq\Hi$, the limiting normal cone coincides with the classical convex normal cone
$$\left\{ u\in\Hi : \langle c-x, u\rangle \leq 0, \forall c\in C\right\}.$$
In addition, when $C$ is a linear subspace, its normal cone turns to its orthogonal complement; i.e., $N_C(x)=C^\perp$ for all $x\in C$ (see, e.g., \cite[Example~6.43]{BC17}).

As in the case of projectors, there is no general expression relating the normal cone to the intersection of sets to those of the individual sets. The next lemma, which is a key tool in our analysis, establishes such a relation under the same assumptions than \Cref{fact:Pinter}. 

\begin{lemma}\label[lemma]{fact:Normal_interscomp}
Let $C\subseteq\Hi$ be a nonempty and closed set and let $D\subseteq\Hi$ be a linear  subspace. If $P_C(d)\cap D\neq \emptyset$, for all $d\in D$, then 
$$N_{C\cap D}(x)= (N_C(x)\cap D)+D^\perp,\quad  \forall x\in C\cap D.$$

\begin{proof}
Let $x\in C\cap D$ and let $u\in N_{C\cap D}(x)$. Then, there exist $\{\tau_k\}_{k=0}^\infty$, $\{x_k\}_{k=0}^\infty$ and $\{z_k\}_{k=0}^\infty$, with $\tau_k\geq 0$ and $z_k\in P_{C\cap D}(x_k)$, for all $k\in\N$, and $x_k\to x$ such that
\begin{equation}\label{e:u_inter}
u=\lim_{k\to\infty}\tau_k(z_k-x_k).
\end{equation}  
Set $d_k:=P_D(x_k)$, for each $k\in\N$. On the one hand, by continuity of the projector $P_D$ (see \cref{fact:PC}\ref{fact:PC_cont}) we get that $d_k\to P_D(x)=x$. On the other hand, since $z_k\in P_{C\cap D}(x_k)$, by applying \cref{fact:Pinter} we obtain that
\begin{equation}\label{e:u_zk}
z_k\in P_C(d_k)\cap D, \quad \forall k\in \N.
\end{equation}
Now, we can split \eqref{e:u_inter} as
\begin{equation}\label{e:u_decomp}
u=\lim_{k\to\infty}\tau_k(z_k-d_k)+\lim_{k\to\infty}\tau_k(d_k-x_k),
\end{equation}
provided that both limits exist. Indeed, since $P_D$ is a (continuous) linear mapping (see \cref{fact:PC}), from \eqref{e:u_inter} we derive that
\begin{align*}
P_D(u) & =\lim_{k\to\infty} P_D\left(\tau_k(z_k-x_k)\right)\\ 
&=\lim_{k\to\infty} \tau_k(P_D(z_k)-P_D(x_k))\\
&=\lim_{k\to\infty} \tau_k(z_k-d_k),
\end{align*}
where we have used the fact that $z_k\in D$ for all $k\in \N$, according to \eqref{e:u_zk}. This shows that the first limit in \eqref{e:u_decomp}, and therefore both of them, exist. Hence, we have obtained that $u=v+w$ with
\begin{subequations}\label{e:u_vw}
\begin{align}
v&:=P_D(u)=\lim_{k\to\infty}\tau_k(z_k-d_k)\in N_C(x),\\ w&:=\lim_{k\to\infty}\tau_k(d_k-x_k)\in N_D(x)=D^\perp.
\end{align}
\end{subequations}
Since $x\in C\cap D$ was arbitrary, we have proved the direct inclusion
\begin{equation}\label{feq:n2}
N_{C\cap D}(x)\subseteq (N_C(x)\cap D)+D^\perp,\quad  \forall x\in C\cap D.
\end{equation}

On the other hand, by taking $L=D$, $A=C\cap D$ and $a=x\in A$ in \cite[Theorem 3.5(25c)-(25d)]{BLPW13} we obtain that
\begin{equation*}\label{feq:n3}
N_{C\cap D}(x)= (N_{C\cap D}(x)\cap D)+D^\perp.
\end{equation*}
Now, as $C\cap D\subseteq C$, it follows that $N_{C}(x)\subseteq N_{C\cap D}(x)$, which combined with the previous expression yields the reverse inclusion of \eqref{feq:n2} and finishes the proof.
\end{proof}
\end{lemma}

\subsection{Regularity of sets}
Let us finally recall the following notions of regularity of sets and of collection of sets.
\begin{definition}[Super-regular sets]\label[definition]{def:superreg_sets}
A subset $C\subseteq\Hi$ is said to be \emph{super-regular} at a point $\bar x\in C$ if, for any $\varepsilon>0$, there exists $\delta>0$ such that
\begin{equation*}
\langle u, y-z\rangle \leq \varepsilon \|u\|\|y-z\|,\quad \forall y,z\in C\cap \Ball(\bar x;\delta),\, \forall u\in N_C(z).
\end{equation*}
\end{definition}

\pagebreak

\begin{definition}[Regularity of collection of sets]\label[definition]{def:reg_collect_sets}
A finite family of sets $C_1,\ldots,C_r\subseteq\Hi$ is said to be
\begin{enumerate}[label=(\roman*)]
\item \emph{linearly regular} around $\bar x\in\Hi$ if there exist $\kappa\geq 0$ and $\delta>0$ such that
\begin{equation}\label{eq:def_linearlyreg}
d_{\cap_{i=1}^r C_i}(z)\leq \kappa \max\{d_{C_i}(z): i=1,\ldots,r\}, \quad \forall z\in\Ball(\bar x;\delta);
\end{equation}
\item \emph{strongly regular} at $\bar x\in\Hi$ if
\begin{equation}\label{eq:def_stronglyreg}
\sum_{i=1}^r u_i=0 \text{ with } u_i\in N_{C_i}(\bar x), \text{ for } i=1,\ldots,r \;\Longleftrightarrow\; u_1=\cdots=u_r=0.
\end{equation}
\end{enumerate}
\end{definition}

The above conditions are usually employed to derive the local linear convergence of some projection algorithms on nonconvex problems. Recall that a sequence $\{x_k\}_{k\in\N}$ \emph{converges $R$-linearly} to a point $x^*$ if there exist $\eta\in[0,1[$ and $M>0$ such that
\begin{equation*}
\|x_k- x^*\|\leq M\eta^k, \quad \forall k\in\N.
\end{equation*}
In the following fact, we recall the (local) linear convergence of the so-called {\em generalized Douglas--Rachford} algorithm for two sets under regularity conditions.

\begin{fact}[Linear convergence of gDR]\label[fact]{fact:gDR}
Let $A,B\subseteq\Hi$ be nonempty super-regular sets of $\Hi$ and let $w\in A\cap B\neq\emptyset$. Let $\lambda,\mu \in{]}0,2]$ and let $\alpha\in{]0,1[}$. Given $x_{0}\in\Hi$, set
\begin{equation}\label{it:factgDR}
x_{k+1}=(1-\alpha)x_{k}+\alpha\left((1-\mu)P_{B}+\mu\Id\right)\circ\left((1-\lambda)P_{A}+\lambda\Id\right)(x_k),\quad\forall k\in\N.
\end{equation}
Suppose that any of the following conditions holds:
\begin{enumerate}[label=(\roman*)]
\item $\{A,B\}$ is strongly regular at $w$,\label{fact:gDR_I}
\item $\min\{\lambda,\mu\}<2$ and $\{A,B\}$ is linearly regular around $w$.\label{fact:gDR_II}\\[-1ex]
\end{enumerate}
If the initial point $x_{0}$ is sufficiently close to $w$, then, the sequence generated by \eqref{it:factgDR} converges $R$-linearly to a point {$x^*\in A\cap B$}. When, in addition, $A$ and $B$ are convex sets, the $R$-linear convergence of the sequence is global.
\begin{proof}
See~\cite[Corollary 5.12]{DP18}.
\end{proof}
\end{fact}

\section{Regularity under a product space reformulation with reduced dimension}\label{sec:ProdSpace}

We begin this section by introducing the product space reformulation in a reduced dimensional product space proposed in \cite{campoy21}. To this aim, consider the product space $$\Hi^{r-1}=\Hi\times\stackrel{(r-1)}{\cdots}\times\Hi,$$ endowed with the inner product
\begin{equation*}
\langle \bs x, \bs y \rangle:=\sum_{i=1}^{r-1}\langle x_i,y_i\rangle, \quad \forall \bs{x}=(x_1,\ldots,x_{r-1}),\bs{y}=(y_1,\ldots,y_{r-1})\in\Hi^{r-1},
\end{equation*}
and define
$$\bs D_{r-1}:= \left\{ (x,\ldots,x)\in \Hi^{r-1} : x\in\Hi \right\},$$
which is a linear subspace of $\Hi^{r-1}$ commonly known as the \emph{diagonal}. We denote $\bs{j}_{r-1}:\Hi\to\bs{D}_{r-1}$ the canonical embedding that maps any $x\in\Hi$ to $\bs{j}_{r-1}(x)=(x,\ldots,x)\in\bs{D}_{r-1}$. Then, consider the product sets
\begin{subequations}\label{eq:npsf_sets}
\begin{align}
\bs B:= & C_1\times\cdots\times C_{r-1} \subseteq \Hi^{r-1},\label{eq:npsf_A}\\[0.2em]
\bs K:= &  C_r^{r-1}\cap \bs D_{r-1}= \{(x,\ldots,x)\in\Hi^{r-1}: x\in C_r\}\subseteq \Hi^{r-1}.\label{eq:npsf_K}
\end{align}
\end{subequations}
The equivalency, from the point of view of projection algorithms, between problem \eqref{eq:FP} and the one described by the sets in \eqref{eq:npsf_sets} is recalled in the following fact.

\begin{fact} Let $\allowbreak C_1,C_2,\ldots,$ $C_r\subseteq \Hi$ be closed sets and let $\bs B,\bs K\subseteq\Hi^{r-1}$ the product operators as defined in \eqref{eq:npsf_sets}. Then the following hold.\label[fact]{p:npsf}
\begin{enumerate}[label=(\roman*)]
\item $\bs B$ is closed and \label{p:npsf_B}
$$P_{\bs B}(\bs x)=P_{C_1}(x_1)\times \cdots\times P_{C_{r-1}}(x_{r-1}),\quad \forall\bs x=(x_1,\ldots,x_{r-1})\in \Hi^{r-1}.$$
If, in addition, $C_1,\ldots,C_{r-1}$ are convex then so is $\bs B$.
\item $\bs K$ is closed and \label{p:npsf_K}
$$P_{\bs K}(\bs x)=\bs j_{r-1}\left(P_{C_r}\left(\frac{1}{r-1} \sum_{i=1}^{r-1} x_i\right)\right),\quad \forall\bs x=(x_1,\ldots,x_{r-1})\in \Hi^{r-1}.$$
If, in addition, $C_{r}$ is convex then so is $\bs K$.
\item $\bs B\cap \bs K=\bs j_{r-1}\left(\cap_{i=1}^{r} C_i\right)$. \label{p:npsf_FP}
\item $P_{\bs B\cap \bs K}(\bs x)=\bs j_{r-1}\left(P_{\cap_{i=1}^r C_i}(x)\right), \quad\forall\bs{x}=\bs{j}_{r-1}(x)\in\bs D_{r-1}$. \label{p:npsf_BP}
\end{enumerate}
\begin{proof}
See \cite[Proposition~4.3]{campoy21}.
\end{proof}
\end{fact}

\begin{remark}
Recall that classical Pierra's reformulation~\cite{Pierra} reframes problem \eqref{eq:FP} as  
\begin{equation*}\label{eq:Pierra}
\text{Find } \bs x\in \bs C \cap \bs D_r\subseteq\Hi^{r}, \quad \text{with } \bs C:=C_1\times C_2\times\cdots\times C_r,
\end{equation*}
where the projectors are given by
\begin{equation*}\label{eq:pierra_proj}
P_{\bs C}(\bs x):= \prod_{i=1}^r P_{C_i}(x_i), \quad
P_{\bs D_{r}}:= \bs j_{r}\left(\frac{1}{r}\sum_{i=1}^r x_i\right);
\end{equation*}
for any $\bs x=(x_1,x_2,\ldots,x_r)\in\Hi^r$ (see, e.g., \cite[Proposition~4.2]{campoy21}). In contrast, we reformulate the same feasibility problem as
\begin{equation*}
\text{Find } \bs x\in \bs B\cap \bs K\subseteq\Hi^{r-1},
\end{equation*}
with $\bs B$ and $\bs K$ being the sets in \eqref{eq:npsf_sets}, whose projectors are given in \Cref{p:npsf}. Note that this leads to a reduction of one dimension of the ambient space in comparison to Pierra's trick. The computational advantage was shown in \cite{campoy21} through some numerical experiments.
\end{remark}

The analysis of the regularity properties of the sets in Pierra's reformulation is usually employed in order to derive local linear convergence of parallel projection algorithms. See, for instance, \cite[Theorem~7.3]{LLM09} for the \emph{method of averaged projections}. In particular, super-regularity of sets and linear and strong regularity of the intersection are kept after the reformulation (see, e.g., \cite[Propositions~3.1(i) and 3.2]{DDHT21}). We establish next analogous results for the reformulation in the product space with reduced dimension in \Cref{p:npsf}. Although our analysis employs similar techniques to those of \cite{DDHT21}, we need to stablish first the following technical result about the normal cones to the product sets in \eqref{eq:npsf_sets}.

\begin{lemma}\label[lemma]{lem:ncps}
Let $C_1,\ldots,$ $C_r\subseteq \Hi$ be nonempty and closed sets, let $\bs B,\bs K\subseteq \Hi^{r-1}$ be the product sets as defined in \eqref{eq:npsf_sets} and consider $\bs D_{r-1}$ the diagonal of the product space $\Hi^{r-1}$. Then, the following hold.
\begin{enumerate}[label=(\roman*)]
\item $N_{\bs B}(\bs x)=N_{C_1}(x_1)\times \cdots \times N_{C_{r-1}}(x_{r-1})$, $\forall \bs x=(x_1,\ldots,x_{r-1})\in\Hi^{r-1}$.\label{lem:ncps_I}
\item $N_{\bs D_{r-1}}(\bs x)=\bs D_{r-1}^\perp$, for all $\bs x=\bs j_{r-1}(x)\in\bs D_{r-1}$, and\label{lem:ncps_II}
$$\bs D_{r-1}^\perp=\left\{(u_1,\ldots,u_{r-1})\in\Hi^{r-1} : \sum_{i=1}^{r-1} u_i=0\right\}.$$
\item $N_{\bs K}(\bs x)=\bs j_{r-1}\left(N_{C_r}(x)\right)+\bs D_{r-1}^\perp$, $\forall \bs x=\bs j_{r-1}(x)\in\bs K$.\label{lem:ncps_III}
\end{enumerate}
\begin{proof}
The proofs of \ref{lem:ncps_I} and \ref{lem:ncps_II} follow from, e.g., \cite[Proposition 1.2]{MorI} and \cite[Proposition~26.4(i)-(ii)]{BC17}, respectively. To prove \ref{lem:ncps_III} let us define
\begin{equation}\label{eq:S}
\bs S:=C_r\times\stackrel{(r-1)}{\cdots}\times C_r\subseteq \Hi^{r-1},
\end{equation}
and fix any $\bs x=\bs j_{r-1}(x)\in \bs K=\bs S\cap \bs D_{r-1}$. By \cref{p:npsf}\ref{p:npsf_B}, for any arbitrary point $\bs q=\bs j_{r-1}(q)\in\bs D_{r-1}$, it holds that
$$\bs j_{r-1}(p)\in P_{\bs S}(\bs q)\cap \bs D_{r-1},\quad \forall  p\in P_{C_r}(q).$$
In particular, $P_{\bs S}(\bs q)\cap \bs D_{r-1}\neq\emptyset$ for all $\bs q \in \bs{D}_{r-1}$. Hence, we can apply \Cref{fact:Normal_interscomp} to express
$$N_{\bs K}(\bs x)=N_{\bs S\cap \bs D_{r-1}}(\bs x)=(N_{\bs S}(\bs x)\cap D_{r-1}) +D_{r-1}^\perp.$$
Since $N_{\bs S}(\bs x)=\prod_{i=1}^{r-1}N_{C_r}(x)$ by item \ref{lem:ncps_I}, the result follows.
\end{proof}
\end{lemma}

We are now ready to derive our main result regarding the regularity of the product sets in \eqref{eq:npsf_sets} and of their intersection, provided that the original sets verify those conditions.

\begin{theorem}\label{th:reg}
Let $C_1,\ldots,$ $C_r\subseteq \Hi$ be nonempty and closed sets and let $\bs B,\bs K\subseteq \Hi^{r-1}$ be the product sets as defined in \eqref{eq:npsf_sets}. Then, the following statements hold.
\begin{enumerate}[label=(\roman*)]
\item If $C_i$ is super-regular at $\bar x_i\in C_i$, for all $i=1,\ldots,r-1$, then the product set $\bs B$ is super-regular at $\bs {\bar x}:=(\bar x_1,\ldots,\bar x_{r-1})\in \bs B$.\label{th:regI}
\item If $C_r$ is super-regular at $\bar x\in C_r$, then $\bs K$ is super-regular at $\bs {\bar x}:=\bs j_{r-1}(\bar x)\in \bs K$.\label{th:regII}
\item If $\{C_1,\ldots,C_r\}$ is linearly regular around $\bar  x\in\Hi$, then $\{\bs B,\bs K\}$ is linearly regular around $\bs {\bar x}:=\bs j_{r-1}(\bar x)\in\Hi^{r-1}$.\label{th:regIII}
\item If $\{C_1,\ldots,C_r\}$ is strongly regular at $\bar x\in\Hi$, then $\{\bs B,\bs K\}$ is strongly regular at $\bs {\bar x}:=\bs j_{r-1}(\bar x)\in\Hi^{r-1}$.\label{th:regIV}
\end{enumerate}

\begin{proof}
\ref{th:regI}: See, e.g., \cite[Proposition 3.1(a)]{DDHT21}.

\ref{th:regII}: Suppose that $C_r$ is super-regular at $\bar x\in C_r$, set $\bs {\bar x}:=\bs j_{r-1}(x)\in \bs K$ and pick any arbitrary $\varepsilon >0$. Consider the set $\bs S$ as in \eqref{eq:S} so that $\bs K=\bs S\cap \bs D_{r-1}$. From item \ref{th:regI} we get that $\bs S$ is super-regular at $\bs {\bar x}$ and, thus, there exists $\delta>0$ such that
\begin{equation}\label{e:srS}
\langle \bs v, \bs y-\bs z\rangle \leq \varepsilon \|\bs v\|\|\bs y-\bs z\|,\quad \forall \bs y,\bs z\in \bs S\cap \Ball(\bs {\bar x};\delta),\, \forall \bs v\in N_{\bs S}(\bs z).
\end{equation}
Let $\bs y,\bs z\in \bs K\cap \Ball(\bs {\bar x};\delta)$ and let $\bs u\in N_{\bs K}(\bs z)$. By \cref{lem:ncps}\ref{lem:ncps_III} we can express
\begin{equation*}
\bs u=\bs v + \bs w, \quad \text{with } \bs v\in N_{\bs S}(\bs z)\cap \bs D_{r-1} \text{ and } \bs w\in {\bs D_{r-1}}^\perp.
\end{equation*}
In view of \eqref{e:srS} we get that
\begin{equation}\label{e:srS_bis}
\langle \bs v, \bs y-\bs z\rangle \leq \varepsilon \|\bs v\|\|\bs y-\bs z\|.
\end{equation}
Further, it holds that
\begin{equation}\label{e:srD}
\langle \bs w, \bs y-\bs z\rangle = 0,
\end{equation}
as $\bs w\in {\bs D_{r-1}}^\perp$ and $\bs y,\bs z\in\bs D_{r-1}$.
In addition, $\|\bs u\|^2=\|\bs v\|^2+\|\bs w\|^2$ since $\bs v\in \bs D_{r-1}$ and $\bs w\in \bs D^\perp_{r-1}$. In particular, this implies that $\|\bs v\|\leq \|\bs u\|$, which combined with   \eqref{e:srS_bis} and \eqref{e:srD} yields
\begin{align*}
\langle \bs u, \bs y-\bs z\rangle&=\langle \bs v, \bs y-\bs z\rangle + \langle \bs w, \bs y-\bs z\rangle\\
& \leq  \varepsilon \|\bs v\|\|\bs y-\bs z\|\\
& \leq  \varepsilon \|\bs u\|\|\bs y-\bs z\|
\end{align*}
and proves that $\bs K$ is super-regular at $\bs {\bar x}$.

\ref{th:regIII}: First of all we note that, for any $\bs q=\bs j_{r-1}(q)\in\bs D_{r-1}$, thanks to \cref{p:npsf} one can easily check that
\begin{subequations}\label{e:dprodsets}
\begin{align}
d^2_{\bs B\cap \bs K}(\bs q)&=(r-1)d^2_{\cap_{i=1}^rC_i}(q),\label{e:dprodsets_BK}\\
d^2_{\bs B}(\bs q)&=\sum_{i=1}^{r-1} d^2_{C_i}(q),\label{e:dprodsets_B}\\
d^2_{\bs K}(\bs q)&=(r-1)d^2_{C_r}(q).\label{e:dprodsets_K}
\end{align}
\end{subequations}
In particular, \eqref{e:dprodsets_B} implies that $d_{C_i}(q)\leq d_{\bs B}({\bs q})$, for all $i=1,\ldots,r-1$, which combined with \eqref{e:dprodsets_K} yields to
\begin{equation}\label{e:maxleqmax}
\begin{aligned}
\max\{d_{C_i}(q): i=1,\ldots,r\}&\leq\max\left\{d_{\bs B}(\bs q),\frac{1}{\sqrt{r-1}}d_{\bs K}(\bs q)\right\}\\&\leq \max\left\{d_{\bs B}(\bs q),d_{\bs K}(\bs q)\right\}.
\end{aligned}
\end{equation}
Now, suppose that $\{C_1,\ldots,C_r\}$ is linearly regular around $\bar x\in\Hi$; i.e., there exits $\kappa>0$ and $\delta>0$ such that \eqref{eq:def_linearlyreg} holds, and set $\bs {\bar x}:=\bs j_{r-1}(\bar x)\in\Hi^{r-1}$. Pick any $\bs z\in \Ball(\bs {\bar x}; \sqrt{r-1}\tfrac{\delta}{2})$ and let $\bs q=P_{\bs D_{r-1}}(\bs z)$. We trivially get that
\begin{equation}\label{e:z-q}
\|\bs q-\bs z\| = d_{\bs D_{r-1}}(\bs z)\leq d_{\bs K}(\bs z).
\end{equation}
Moreover, since $\bs {\bar x}\in \bs D_{r-1}$ we have that
\begin{equation*}
\|\bs q-\bs {\bar x}\|\leq \|\bs q-\bs z\|+\|\bs {\bar x}-\bs z\|\leq 2\|\bs {\bar x}-\bs z\|\leq \sqrt{r-1}\delta,
\end{equation*}
which implies that $\bs q=\bs j_{r-1}(q)$ with $q\in\Ball(x;\delta)$. Therefore, the linear regularity of $\{C_1,\ldots,C_r\}$ around $\bar x$ implies that 
\begin{equation}\label{e:dinter}
d_{\cap_{i=1}^r C_i}(q)\leq \kappa \max\{d_{C_i}(q): i=1,\ldots,r\}.
\end{equation}
Hence, using \eqref{e:dprodsets_BK},  \eqref{e:maxleqmax}, \eqref{e:z-q} and \eqref{e:dinter} we deduce that  
\begin{align*}
d_{\bs B\cap \bs K}(\bs z) & \leq d_{\bs B\cap \bs K}(\bs q) + \|\bs q-\bs z\|\\
& = \sqrt{r-1}d_{\cap_{i=1}^rC_i}(q) + d_{\bs K}(\bs z)\\
& \leq  \kappa\sqrt{r-1}\max\{d_{C_i}(q): i=1,\ldots,r\} + d_{\bs K}(\bs z)\\
& \leq \kappa\sqrt{r-1}\max\left\{d_{\bs B}(\bs q),d_{\bs K}(\bs q)\right\} + d_{\bs K}(\bs z)\\
& \leq \kappa\sqrt{r-1}\left(\max\left\{d_{\bs B}(\bs z),d_{\bs K}(\bs z)\right\}+\|\bs q-\bs z\|\right) + d_{\bs K}(\bs z)\\
& \leq \kappa\sqrt{r-1}\max\left\{d_{\bs B}(\bs z),d_{\bs K}(\bs z)\right\} + \left(1+\kappa\sqrt{r-1}\right)d_{\bs K}(\bs z)\\
& \leq \left(1+2\kappa\sqrt{r-1}\right)\max\left\{d_{\bs B}(\bs z),d_{\bs K}(\bs z)\right\},
\end{align*}
which shows that $\{\bs B,\bs K\}$ is linearly regular around $\bs {\bar x}$.

\ref{th:regIV}: Suppose that $\{C_1,\ldots,C_r\}$ is strongly regular at $\bar x\in\Hi$ and set $\bs {\bar x}=\bs j_{r-1}(\bar x)\in \bs D_{r-1}$. Let $\bs v\in N_{\bs B}(\bs {\bar x})$ and $\bs w\in N_{\bs K}(\bs {\bar x})$ such that
\begin{equation}\label{e:v+w}
\bs v+\bs w= 0.
\end{equation}
By \cref{lem:ncps}\ref{lem:ncps_III} we can write $\bs w=\bs z + \bs u$ where $\bs z\in \bs j_{r-1}\left(N_{C_r}(\bar x)\right)$ and $\bs u\in {\bs D_{r-1}}^\perp$. Now, in view of \cref{lem:ncps}\ref{lem:ncps_I}--\ref{lem:ncps_II} we obtain that
\begin{subequations}\label{e:vzu_dec}
\begin{align}
\bs v=(v_1,\ldots,v_{r-1}),&\quad \text{with } v_i\in N_{C_i}(\bar x), \forall i=1,\ldots,r-1,\\
\bs z=\bs j_{r-1}(z),&\quad \text{with } z\in N_{C_r}(\bar x),\\
\bs u=(u_1,\ldots,u_{r-1}),&\quad \text{with } \sum_{i=1}^{r-1}u_i=0.
\end{align}
\end{subequations}
Then, by combining \eqref{e:v+w} and \eqref{e:vzu_dec} we get that $v_i+z+u_i=0$ for all $i\in\{1,\ldots,r-1\}$. By summing up all these equations we arrive at
\begin{equation*}\label{e:v+w_reg}
0=\sum_{i=1}^{r-1}v_i + (r-1)z + \sum_{i=1}^{r-1}u_i=\sum_{i=1}^{r-1}v_i + (r-1)z.
\end{equation*}
Since $v_i\in N_{C_i}(\bar x)$ for $i=1,\ldots,r-1$, and $(r-1)z\in N_{C_r}(\bar x)$, the strong regularity of $\{C_1,\ldots,C_r\}$ around $\bar x$ implies that $v_1=\cdots=v_{r-1}=z=0$. Therefore, $\bs v=\bs w=0$ and we get that $\{\bs B,\bs K\}$ is strongly regular at $\bs {\bar x}$.
\end{proof}
\end{theorem}

\section{Application to projection algorithms}\label{sec:applic}

We finally apply our main result (\Cref{th:reg}) to derive (local) linear convergence of projection algorithms constructed by means of the product space reformulation in \Cref{p:npsf}. In particular, we consider the \emph{generalized Douglas--Rachford (gDR) algorithm} analyzed in \cite{DP18}, as it includes the \emph{method of alternating projections (MAP)} and the \emph{Douglas--Rachford (DR) algorithm} as particular cases.

\begin{theorem}[Linear convergence of parallel gDR algorithm with reduced lifting]\label{th:gDR}
Let $\allowbreak C_1, C_2, \ldots, C_r\subseteq\Hi$ be nonempty super-regular sets with $\cap_{i=1}^r C_i\neq\emptyset$. Let $\lambda,\mu \in{]}0,2]$ and let $\alpha\in{]0,1[}$. Given $x_{1,0},\ldots,x_{r-1,0}\in\Hi$, set
\begin{equation}\label{it:DR}
\begin{aligned}
&\text{for } k\in\N:\\
&\left\lfloor\begin{array}{l}
p_k=P_{C_r}\left(\frac{1}{r-1}\sum_{i=1}^{r-1} x_{i,k}\right),\\
\text{for } i=1, 2,\ldots, r-1:\\
\left\lfloor\begin{array}{l}
u_{i,k}=(1-\lambda)x_{i,k}+\lambda p_k,\\
z_{i,k}=P_{C_i}\left(u_{i,k}\right),\\
x_{i,k+1}=(1-\alpha)x_{i,k}+\alpha\left((1-\mu)u_{i,k}+\mu z_{i,k}\right).
\end{array}\right.
\end{array}\right.
\end{aligned}
\end{equation}
Suppose that any of the following conditions holds:
\begin{enumerate}[label=(\roman*)]
\item $\{C_1,C_2,\ldots,C_r\}$ is strongly regular at a point $\bar x\in\cap_{i=1}^r C_i$,\label{th:gDR_I}
\item $\min\{\lambda,\mu\}<2$ and $\{C_1,C_2,\ldots,C_r\}$ is linearly regular around $\bar x\in\cap_{i=1}^r C_i$.\label{th:gDR_II}\\[-1ex]
\end{enumerate}
If the initial points $x_{1,0},\ldots,x_{r-1,0}$ are sufficiently close to $\bar x$, then, for each ${i\in\{1,\ldots,r-1\}}$, the sequence $\{x_{i,k}\}_{k\in\N}$ converges $R$-linearly to a point {$x^*\in\cap_{i=1}^r C_i$}. When, in addition, $C_1,C_2,\ldots,C_r$ are convex sets, the $R$-linear convergence of the sequences is global.
\end{theorem}
\begin{proof}
Consider the product Hilbert space $\Hi^{r-1}$ and let $\bs B,\bs K\subseteq\Hi^{r-1}$ be the product sets defined in \eqref{eq:npsf_sets} which, in view of \Cref{th:reg}\ref{th:regI}--\ref{th:regII}, are super-regular. By \Cref{p:npsf}\ref{p:npsf_FP} we get that $\bs B\cap \bs K=\bs j_{r-1}(\cap_{i=1}^r C_i)\neq\emptyset$. Set $\bs x_{k}:=(x_{1,k},\ldots,x_{r-1,k})\in\Hi^{r-1}$ for all $k\in\N$. Hence, according to \Cref{p:npsf}\ref{p:npsf_B}--\ref{p:npsf_K}, we can rewrite \eqref{it:DR} as 
\begin{equation}\label{it:DR_nps}
\bs x_{k+1}=(1-\alpha)\bs x_{k}+\alpha\left((1-\mu)P_{\bs B}+\mu\Id\right)\circ\left((1-\lambda)P_{\bs K}+\lambda\Id\right)(\bs x_k),\quad\forall k\in\N.
\end{equation}
Note that \ref{th:gDR_I} (resp. \ref{th:gDR_II})  implies that $\{\bs B,\bs K\}$ is strongly regular (resp. linearly regular) at $\bs{\bar x}=\bs j_{r-1}(\bar x)\in \bs B\cap \bs K$ according to \Cref{th:reg}\ref{th:regIII} (resp. \Cref{th:reg}\ref{th:regIV}). Hence, the result follows from \cref{fact:gDR}.
\end{proof}

As previously mentioned, iteration \eqref{it:DR_nps} recovers some well-known classical projection methods. Hence, \Cref{th:gDR} provides local linear convergence for reduced parallel versions of these algorithms. We state next such result for the method of alternating projections, leading to what we will refer to as \emph{reduced averaged projections method}.

\begin{corollary}[Linear convergence of reduced averaged projections method]\label[corollary]{cor:RAP}
Let $\allowbreak C_1, C_2, \ldots, C_r\subseteq\Hi$ be nonempty super-regular sets with linearly regular intersection around $\bar x\in\cap_{i=1}^r C_i$. Given $x_{0}\in\Hi$, set
\begin{equation}\label{cor:itRAP}
x_{k+1}=P_{C_r}\left(\frac{1}{r-1}\sum_{i=1}^{r-1} P_{C_i}(x_k)\right), \quad \forall k\in\N.
\end{equation}
If $x_{0}$ is sufficiently close to $\bar x$, then the sequence $\{x_{k}\}_{k\in\N}$ converges $R$-linearly to a point {$x^*\in\cap_{i=1}^r C_i$}. When, in addition, $C_1,C_2,\ldots,C_r$ are convex sets, the $R$-linear convergence of the sequence is global.
\begin{proof}
Apply \Cref{th:gDR}\ref{th:gDR_I} with $\lambda=\mu=\alpha=1$.
\end{proof}
\end{corollary}

\begin{remark}
An analogous result can be derived for the Douglas--Rachford algorithm by taking $\lambda=\mu=2$ in \eqref{it:DR}.  In particular, \Cref{th:gDR} under scenario \ref{th:gDR_I} applies to the parallel DR-algorithm with reduced dimension proposed in \cite[Theorem~5.1]{campoy21} in the context of feasibility problems.
\end{remark}

\subsection{Numerical experiment}

In this section, we present a numerical example to illustrate the linear convergence of the reduced averaged projections method discussed in \Cref{cor:RAP}. Our objective is to replicate the signal compression problem analyzed in \cite[Section~9]{LLM09}, which was utilized to show the linear convergence of the traditional averaged projections method.

Given a ``dictionary'' $W\in\Rr^{n\times m}$ and a threshold $\alpha>0$, the recovery of the signal is addressed by solving the feasibility problem
\begin{equation*}
\text{find } \quad U^*\in L\cap M\cap C \subseteq\Rr^{d\times m},
\end{equation*}
where
\begin{subequations}\label{e:num_sets}
\begin{align}
L &:= \left\{ U\in\Rr^{d\times m} : U=PW, \text{ with } P\in\Rr^{d\times n} \right\},\label{setL}\\
M &:= \left\{ U\in\Rr^{d\times m} : UU^T = I \right\},\label{setM}\\
C &:= \left\{ U\in\Rr^{d\times m} : \|U\|_{\infty}\leq \alpha \right\}.\label{setC}
\end{align}
\end{subequations}
It is not difficult to check that the projectors onto these sets can be computed~as
\begin{align*}
P_L(U) &= UW^T(WW^T)^{-1}W,\\
P_M(U) &= \left\{PQ^T : P\Sigma Q^T \text{ is a singular value decomposition of } U \right\},\\
P_C(U) &= \min(\max(U,-\alpha),\alpha);
\end{align*}
where the maximum and minimum in $P_C$ are understood componentwise. Furthermore, as mentioned in \cite{LLM09}, the three sets in \eqref{e:num_sets} are super-regular, whereas the linear regularity of their intersection is expected from randomness when generating the problem, provided that $\alpha$ is not too small. 

In our experiment, we set $n=128$, $m=512$, $d=8$ and $\alpha=0.1$. The entries of the matrix $W\in\Rr^{128\times 512}$, as well as those of the initial iterate $U_0\in\Rr^{8\times 512}$, were randomly generated from a standard normal distribution. From that point, we run the averaged projections algorithm, which iterates as
\begin{equation}\label{eq:avproj}
U_{k+1}=\frac{1}{3}\left(P_L(U_k)+P_M(U_k)+P_C(U_k)\right), \quad\forall k\in\N;
\end{equation} 
and the reduced averaged projections method in \Cref{cor:RAP}. Note that \eqref{eq:avproj} is completely symmetric with respect to the order of the sets. However, this is not the case for the reduced averaged projections in \eqref{cor:itRAP}, where the set $C_r$ acts as a ``central coordinator''. Thus, in our experiment we consider all three possibilities for this method depending on which of the sets $L$, $M$, or $C$ plays the role of coordinator (indicated between brackets). We stopped each algorithm when $\|U_{k+1}-U_k\|<10^{-12}$. In \Cref{fig:exp} we plot the norm $\|U_k-U^*\|$  with respect to the iteration, where $U^*$ denotes the limit of the sequence.

\begin{figure}[htp!]\centering
\includegraphics[width=0.71\textwidth]{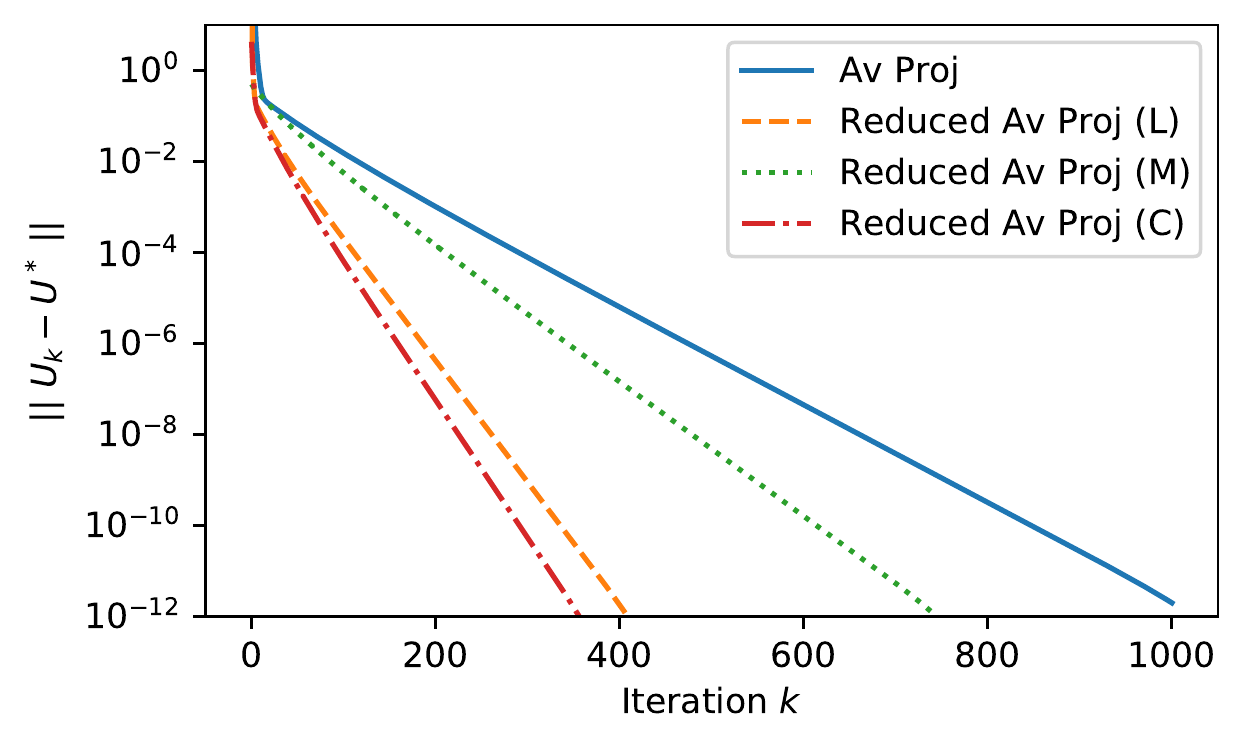}
\caption{Comparison of the convergence rate of \emph{averaged projections} and \emph{reduced averaged projections} methods for solving a signal compression problem. For each method we plot the distance to solution, in logarithmic scale, with respect to the iteration.}\label{fig:exp}
\end{figure}

We can clearly observe a linearly convergent behavior of all tested methods, showing the reduced versions of the method a better convergence rate than its classical version. Furthermore, the choice of the coordinator set $C_r$ in \eqref{cor:itRAP} seems to have a strong impact in the convergence rate of the method. In our experiment, the fastest convergence was achieved by selecting set $C$ in \eqref{setC}, followed closely by selecting $L$ in \eqref{setL}. Overall, the results suggest that choosing the appropriate coordinator set can significantly improve the convergence rate of the method.

\section{Conclusions}\label{sec:Conc}

In this manuscript we explored how some regularity properties of sets and of collections of sets are preserved under a reformulation in a product space with reduced dimension. This allows for the establishment of local linear convergence of parallel projection methods constructed through this reformulation. Specifically, the results were applied to the generalized Douglas--Rachford algorithm, which include some well-known projection algorithms as particular~cases. 

A numerical demonstration on a signal compression problem, replicating that of \cite[Section~9]{LLM09}, was included. This study tested the method of averaged projections and some reduced versions of this method constructed trough the analyzed reformulation. As expected, all methods showed to be linearly convergent. In addition, a better convergence rate was obtained for the reduced methods in this specific experiment. It remains open for future research to analyze the convergence rate of these methods, particularly with respect to the order of the sets and its effect on the rate. 

{\small
\paragraph{\small Acknowledgements}
The author was partially supported by the Ministry of Science and Innovation of Spain and the European Regional Development Fund (ERDF) of the European Commission, Grant PID2022-136399NB-C21, and by the Generalitat Valenciana (AICO/2021/165).

}


\begin{thebibliography}{99}

\setlength{\itemsep}{-3pt plus 0ex}

\bibitem{aragonbot21}
	Arag\'on-Artacho, F.J., Boţ, R.I., Torregrosa-Bel\'en, D.:
	\newblock A primal-dual splitting algorithm for composite monotone inclusions with minimal lifting.
	\newblock {\em Numer. Alg.} 93(1), 103-130 (2023) 
  
\bibitem{aragon2020ORclassroom}
  Arag\'on-Artacho, F.J., Campoy, R., Tam, M.K.:
  \newblock The Douglas--Rachford algorithm for convex and nonconvex feasibility problems.
  \newblock {\em Math. Methods Oper. Res.} 91(2), 201--240 (2020)

\bibitem{aragonmalitsky21}
	Arag\'on-Artacho, F.J., Malitsky, Y., Tam, M.K., Torregrosa-Bel\'en, D.: 
	\newblock Distributed Forward-Backward methods without central coordination.
	\newblock  {\em Comput. Opt. Appl.} (2022). DOI: \href{https://doi.org/10.1007/s10589-022-00400-z}{10.1007/s10589-022-00400-z}

\bibitem{BC17}
	Bauschke, H.H., Combettes, P.L.:
	\newblock {\em Convex analysis and monotone operator theory in Hilbert spaces, 2nd ed.}
	\newblock Springer, Berlin (2017)

\bibitem{BLPW13}
	Bauschke, H.H., Luke, D.R., Phan, H.M., Wang, X.:
	\newblock Restricted normal cones and the method of alternating projections: theory.
	\newblock {\em Set-Valued Var. Anal.} 21(3), 431--473 (2013)

\bibitem{BNP15}
	Bauschke, H.H., Noll, D., Phan, H.M.:
	\newblock Linear and strong convergence of algorithms involving averaged nonexpansive operators.
	\newblock {J. Math. Anal. Appl.} 421(1), 1--20 (2015)
	
\bibitem{BPW14}
	Bauschke, H.H., Phan, H.M., Wang, X.:
	\newblock The method of alternating relaxed projections for two nonconvex sets.
	\newblock {\em Vietnam J. Math.} 42(4), 421--450 (2014)

\bibitem{BSW22}
	Bauschke, H.H., Singh, S., Wang, X.:
	\newblock The splitting algorithms by Ryu, by Malitsky--Tam, and by Campoy applied to normal cones of linear subspaces converge strongly to the projection onto the intersection. 
	\newblock {\em SIAM J. Optim.} 33(2), 739--765 (2023)

\bibitem{MT14}
	Borwein, J.M., Tam, M.K.:
	A cyclic Douglas–Rachford iteration scheme.
	\newblock {\em J. Optim. Theory Appl.} 160(1), 1--29 (2014)
	
\bibitem{B65}
	Bregman, L.:
	\newblock The method of successive projection for finding a common point of convex sets.
	\newblock {\em Sov. Math. Dokl.} 6, 688--692 (1965)

\bibitem{bricenoarias}
	Brice\~no-Arias, L.:
	\newblock Resolvent splitting with minimal lifting for composite monotone inclusions.
	\newblock ArXiv preprint (2021): \href{https://arxiv.org/abs/2111.09757}{2111.09757}

\bibitem{campoy21}
	Campoy, R.:
	\newblock {A product space reformulation with reduced dimension for splitting algorithms.}
	\newblock {\em Comput. Optim. Appl.} 83(1), 319---348 (2022)
  
\bibitem{D01}
	Deutsch, F.:
	\newblock {\em Best Approximation in Inner Product Spaces.}
	\newblock CMS Books in Mathematics/Ouvrages de Math\'ematiques de la SMC. Springer-Verlag, New York (2001)

\bibitem{DDHT21}
  Dao, M., Dizon, N., Hogan, J., Tam, M.K.:
  \newblock Constraint reduction reformulations for projection algorithms with applications to wavelet construction.
  \newblock {\em J. Optim. Theory Appl.} 190(1), 201--233 (2021)

\bibitem{DP18}
	Dao, M.N., Phan, H.M.:
	\newblock Linear convergence of the generalized Douglas-Rachford algorithm for feasibility problems.
	\newblock {\em J. Global Optim.} 72(3), 443--474 (2018)


\bibitem{DP19}
	Dao, M.N., Phan, H.M.:
	\newblock Linear convergence of projection algorithms.
	\newblock {\em Math. Oper. Res.} 44(2), 715--738 (2019)


\bibitem{DR56}
	Douglas, J., Rachford, H.H.:
	\newblock On the numerical solution of heat conduction problems in two and three space variables.
	\newblock {\em Trans. Amer. Math. Soc.} 82(2), 421--439 (1956)

\bibitem{HL13}
	Hesse, R., Luke, D.R.:
	\newblock Nonconvex notions of regularity and convergence of fundamental algorithms for feasibility problems.
	\newblock {\em SIAM J. Optim.} 23(4), 2397--2419 (2013)
	
\bibitem{K06}
	 Kruger, A.Y.:
	 \newblock About regularity of collections of sets.
	 \newblock {\em Set-Valued Anal.} 14(2), 187--206 (2006)

\bibitem{KLT18}
	Kruger, A.Y., Luke, D.R., Thao, N.H.:
	\newblock Set regularities and feasibility problems.
	\newblock {\em Math. Program.} 168(1--2), 279--311 (2018)
	
\bibitem{LLM09}
	Lewis, A. S., Luke, D. R., Malick, J.:
	\newblock Local linear convergence for alternating and averaged nonconvex projections.
	\newblock {\em Found. Comput. Math.} 9(4), 485--513 (2009)


\bibitem{LM79}
	Lions, P. L., Mercier, B.:
	\newblock Splitting algorithms for the sum of two nonlinear operators.
	\newblock {\em SIAM J. Numer. Anal.} 16(6), 964--979 (1979)


\bibitem{MT21}
  Malitsky, Y., Tam, M. K.:
  \newblock Resolvent splitting for sums of monotone operators with minimal lifting.
  \newblock {\em Math. Program.}  201(1--2), 231--262 (2023)

\bibitem{MorI}
	Mordukhovich, B.S.:
	\newblock {\em Variational Analysis and Generalized Differentiation I.}
	\newblock Springer, Berlin (2006)

\bibitem{APvN}
	von Neumann, J.:
	\newblock {\em Functional Operators Volume II: The Geometry of Orthogonal Spaces.}
	\newblock Princeton University Press, New Jersey, USA (1950)


\bibitem{Phan16}
	Phan, H.M.: 
	\newblock Linear convergence of the Douglas–Rachford method for two closed sets.
	\newblock {\em Optim.} 65(2), 369--385 (2016)
	
	
\bibitem{Pierra}
  Pierra, G.:
  \newblock Decomposition through formalization in a product space.
  \newblock \emph{Math. Program.} 28(1), 96--115 (1984)
  

\bibitem{ryu}
  Ryu, E. K.:
  \newblock Uniqueness of DRS as the 2 operator resolvent-splitting and impossibility of 3 operator resolvent-splitting.
  \newblock \emph{Math. Program.} 182(1), 233--273 (2020)

\bibitem{Tam22}
	Tam, M.K.:
	\newblock Frugal and decentralised resolvent splittings defined by nonexpansive operators.
	\newblock \emph{Optim. Lett.} 1--19 (2023). DOI: \href{https://doi.org/10.1007/s11590-023-02064-y}{10.1007/s11590-023-02064-y}
  
\end{thebibliography}
\end{document}